\documentclass{article}
\usepackage{amsmath,amssymb}
\usepackage{longtable}
\usepackage{fancyhdr}

\usepackage{amsthm}

\newtheorem{theo}{Theorem}[section]
\newtheorem{lemma}[theo]{Lemma}

\newtheorem{col}[theo]{Corollary}

\theoremstyle{definition}
\newtheorem{defn}[theo]{Definition}
\newtheorem{remark}[theo]{Remark}
\newtheorem{ex}[theo]{Example}
\newtheorem{nt}[theo]{Notation}

\def\rank{\mathop{\rm rank}\nolimits}
\def\dim{\mathop{\rm dim}\nolimits}
\def\diag{\mathop{\rm diag}\nolimits}
\def\Ker{\mathop{\rm Ker}\nolimits}

\def\ad{\mathop{\rm ad}\nolimits}

\def\Ann{\mathop{\rm Ann}\nolimits}

\def\Span{\mathop{\rm Span}\nolimits}
\def\C{\mathop{\mathbb{C}}\nolimits}
\def\Z{\mathop{\mathbb{Z}}\nolimits}

\def\g{\mathop{\mathfrak{g}}\nolimits}
\def\h{\mathop{\mathfrak{h}}\nolimits}

\def\Hom{\mathop{\rm Hom}\nolimits}
\def\sl{\mathop{\mathfrak{sl}}\nolimits}

\def\gl{\mathop{\mathfrak{gl}}\nolimits}
\def\V{\mathop{\mathcal{V}}\nolimits}

\newcommand\bigzerou{\smash{\lower.3ex\hbox{\Huge 0}}} 
         
\newcommand\bigstaru{\smash{\lower.3ex\hbox{\Huge $*$}}} 

\numberwithin{equation}{section}

\allowdisplaybreaks[1]
\title{Reduced contragredient Lie algebras and PC Lie algebras}
\vspace{50truept}
\author{ Nagatoshi Sasano}
\begin{document}
\thispagestyle{empty}

\newpage
\begin{center}{\huge  Reduced contragredient Lie algebras and PC Lie algebras \footnote{{\bf 2010 Mathematic Subjects Classification}: Primary 17B65, Secondary 17B67, 17B70}}\end{center}
\vspace{50truept}

\begin{center}{Nagatoshi SASANO}\end{center}
\begin{abstract}
The first aim of this paper is to show that any finite-dimensional reductive Lie algebra and its finite-dimensional completely reducible representation can be embedded into some PC Lie algebra.
The second aim is to find the structure of a PC Lie algebra.
\end{abstract}

\section* {Introduction} 
Using the theory of standard pentads, an arbitrary finite-dimensional reductive Lie algebra and its representation can be embedded into some graded Lie algebra.
The term ``standard pentad'' is defined as the following.
\begin{defn}[\text {standard pentad, \cite [Definition 2.2]{Sa3}}]
Let $\g $ be a Lie algebra, $\rho :\g\otimes V\rightarrow V$ a representation of $\g $ on $V$, $\V $ a submodule of $\Hom (V,\C)$ and $B$ a non-degenerate invariant bilinear form on $\g $ all defined over $\C $.
When a pentad $(\g ,\rho,V,\V,B)$ satisfies the following conditions, we call it a standard pentad:
\begin{itemize}
\item {the restriction of the canonical bilinear form $\langle \cdot ,\cdot \rangle :V\times \Hom (V,\C)\rightarrow \C$ to $V\times \V$ is non-degenerate},
\item {there exists a linear map $\Phi _{\rho }:V\otimes \V\rightarrow \g $, called a $\Phi $-map, satisfying an equation 
$$
B(a,\Phi _{\rho }(v\otimes \phi ))=\langle \rho (a\otimes v),\phi \rangle 
$$
for any $a\in \g$, $v\in V$, $\phi \in \V$.}
\end{itemize}
\end{defn}
From a standard pentad, we can construct a graded Lie algebra.
\begin{theo}[\text {\cite [Theorem 2.15]{Sa3}}]\label {th;stapLie}
For any standard pentad $(\g,\rho,V,\V,B)$, there exists a graded Lie algebra $$L(\g,\rho,V,\V,B)=\bigoplus _{n\in \Z}V_n,$$ called the Lie algebra associated with $(\g,\rho,V,\V,B)$, satisfying the conditions that
$$
V_0\simeq \g
$$
as Lie algebras, that
$$
V_{-1}\simeq \V,\quad V_1\simeq V
$$
as $\g $-modules via the isomorphism of Lie algebras $V_0\simeq \g$ and that the restriction of bracket product $[\cdot ,\cdot ]:V_1\times V_{-1}\rightarrow V_0$ is induced by the $\Phi $-map $\Phi _{\rho }:V\otimes \V \rightarrow \g$.
\end{theo}
When a Lie algebra $\g $ is finite-dimensional, it is known that any pentad $(\g,\rho,V,\V,B)$ is standard.
Thus, we can ``embed'' a finite-dimensional reductive Lie algebra\footnote {It is known that any finite-dimensional reductive Lie algebra has a non-degenerate invariant bilinear form (see \cite [Chapter 1. \S 6.4 Proposition 5]{bu-1}).} and its representation into some graded Lie algebra in the sense of Theorem \ref {th;stapLie}.
Here, we have a problem how to find the structure of a Lie algebra of the form $L(\g, \rho,V,\V,B)$.
In this paper, we shall consider this problem under some assumptions.
\par 
Now, as special cases of standard pentads, we give the notion of pentads of Cartan type.
A pentad of Cartan type is a pentad which has a finite-dimensional commutative Lie algebra and its finite-dimensional diagonalizable representation.
We can describe an arbitrary pentad of Cartan type by two positive integers $r, n$ and three matrices $A, D,\Gamma $ as $P(r,n;A,D,\Gamma )$.
We denote the Lie algebra associated with $P(r,n;A,D,\Gamma )$ by $L(r,n;A,D,\Gamma )$, and moreover, we call a Lie algebra of the form $L(r,n;A,D,\Gamma )$ a PC Lie algebra.
For detail on pentads of Cartan type, see \cite {Sa4}.
\par 
We have two aims of this paper.
The first aim is to show that an arbitrary Lie algebra of the form $L(\g, \rho ,V,\Hom (V,\C),B)$ with a finite-dimensional reductive Lie algebra $\g $ and its finite-dimensional completely reducible representation $(\rho ,V)$ is isomorphic to some PC Lie algebra (Theorem \ref {th;PC}).
And, moreover, the second aim is to find the structure of PC Lie algebras (Theorem \ref {theo;1}).
In \cite [Theorem 3.9]{Sa4},  we have obtained a way how to describe the structure of $L(r,n;A,D,\Gamma )$ under the assumption that $\Gamma \cdot {}^t D\cdot A\cdot D$ is invertible.
In this paper, we shall find the structure of $L(r,n;A,D,\Gamma )$ without any assumptions on $r,n$ and $A,D,\Gamma $.

\begin{nt}
Throughout this paper, all objects are defined over the complex number field $\C$.
We use the following notations:
\begin{itemize}
\item {$\Span (v_1,\ldots ,v_n)$: a vector space spanned by $v_1,\ldots ,v_n$,}
\item {$\mathrm {M}(k,l;\C)$: a set of matrices of size $k\times l$ whose entries belong to $\C$,}
\item {$\diag (c_1,\ldots ,c_m)$: a diagonal matrix of size $m$ whose $(i,i)$-entry is $c_i$,}
\item {$\delta _{ij}$: the Kronecker delta.}
\end{itemize}
\end{nt}
\begin{nt}
We regard a representation $\rho $ of a Lie algebra $\mathfrak{l}$ on $U$ as a linear map 
$$
\rho :\mathfrak{l}\otimes U\rightarrow U
$$
satisfying 
$$
\rho ([a,b]\otimes u)=\rho (a\otimes \rho (b\otimes u))-\rho (b\otimes \rho (a\otimes u))
$$
for any $a,b\in \mathfrak{l}$ and $u\in U$.
Moreover, we denote an ideal $\{a\in \mathfrak{l}\mid \rho (a\otimes u)=0 \ \text {for any}\ u\in U\}$ of $\mathfrak{l}$ by $\Ann U$.
\end{nt}

\section {PC Lie algebras and contragredient Lie algebras}\label {sec;1}
The purpose of this section is to prepare some notion and notations we need to understand the statements of the main theorems, Theorems \ref {th;PC} and \ref {theo;1}.
For detail, refer \cite {ka-1} and \cite {Sa4}.
\begin{defn}[\text {pentads of Cartan type, \cite [Definition 2.4]{Sa4}}]
Let $r$, $n$ be positive integers.
Let $A\in \mathrm {M}(r,r;\C)$ be an invertible square matrix, $D=(d_{ij})\in \mathrm {M}(r,n;\C)$ a matrix and $\Gamma =\diag (\gamma _1,\ldots ,\gamma _n)\in \mathrm {M}(n,n;\C)$ an invertible diagonal matrix.
Let $\h^r=\Span (\epsilon _1,\ldots ,\epsilon _r)$, $\C_D^{\Gamma }=\Span (e_1,\ldots ,e_n)$, $\C_{-D}^{\Gamma }=\Span (f_1,\ldots ,f_n)$ be vector spaces with dimensional $r$, $n$ and $n$ respectively.
We regard $\h ^r$ as a commutative Lie algebra:
$$
\h^r\simeq \gl _1^r
$$
and define representations $\Box _D^r$ and $\Box _{-D}^r$ of $\h ^r$ on $\C_D^{\Gamma }$ and $\C_{-D}^{\Gamma }$ as:
$$
\Box _D^r(\epsilon _i\otimes e_j)=d_{ij}e_j,\quad \Box _{-D}^r(\epsilon _i\otimes f_j)=-d_{ij}f_j
$$
for any $i=1,\ldots ,r$ and $j=1,\ldots ,n$.
Moreover, we define non-degenerate bilinear maps $B_A:\h^r \times \h^r \rightarrow \C$ and $\langle \cdot ,\cdot \rangle _D^{\Gamma }:C_D^{\Gamma }\times \C_{-D}^{\Gamma }\rightarrow \C$ as:
$$
B_A(c_1\epsilon _1+\cdots +c_r\epsilon _r,c_1^{\prime }\epsilon _1+\cdots +c_r^{\prime }\epsilon _r)=\begin{pmatrix}c_1 &\cdots &c_r\end{pmatrix}\cdot {}^{\rm t}A^{-1}\cdot \begin{pmatrix}c_1^{\prime }\\ \vdots \\ c_r^{\prime }\end{pmatrix},\quad \langle e_i,f_j\rangle _D^{\Gamma }=\delta _{ij}\gamma _i
$$
for $i,j=1,\ldots ,n$.
Under these, we define a standard pentad $(\h^r,\Box _D^r,\C_D^{\Gamma },\C_{-D}^{\Gamma },B_A)$ and denote it by $P(r,n;A,D,\Gamma )$.
We call a standard pentad of the form $P(r,n;A,D,\Gamma )$ a pentad of Cartan type.
\end{defn}
\begin{defn}[\text {Cartan matrix of a pentad of Cartan type, \cite [Definition 2.15]{Sa4}}]
For a pentad of Cartan type $P(r,n;A,D,\Gamma )$, put 
$$
C(A,D,\Gamma )=\Gamma \cdot {}^{\rm t}D\cdot A\cdot D.
$$
We call $C(A,D,\Gamma )$ the Cartan matrix of $P(r,n;A,D,\Gamma )$.
\end{defn}
\begin{defn}[\text {PC Lie algebras, \cite [Definition 3.6]{Sa4}}]
For a pentad of Cartan type $P(r,n;A,D,\Gamma )$, we denote its corresponding graded Lie algebra (see Theorem \ref {th;stapLie}) by $L(r,n;A,D,\Gamma )$.
We call a Lie algebra of the form $L(r,n;A,D,\Gamma )$ a PC Lie algebra.
\end{defn}
\begin{remark}
The structure of a PC Lie algebra $L(r,n;A,D,\Gamma )$ is independent to the diagonal matrix $\Gamma $ and to the order of column vectors in $D$ (see \cite [Propositions 1.7 and 2.6]{Sa4}).
\end{remark}
Moreover, we need to recall some notion of graded Lie algebras due to Kac in \cite {ka-1}.
\begin{defn}[transirivity, \text{\cite[p.1275, Definition 2]{ka-1}}]\label{defn;transitive}
A graded Lie algebra 
$$
G=\bigoplus _{i=-\infty }^{+\infty }G_i
$$
is said to be transitive if:
\begin{itemize}
\item {for $x\in G_i$, $i\geq 0$, $[x,G_{-1}]=\{0\}$ implies $x=0$},
\item {for $x\in G_i$, $i\leq 0$, $[x,G_{1}]=\{0\}$ implies $x=0$}.
\end{itemize}
\end{defn}
\begin{defn}[\text {contragredient Lie algebras, \cite [p.1279]{ka-1}}]\label{def;contra}
Let $A=(A_{ij})$ $i,j=1,\ldots ,n$ be a matrix with elements from $\C$.
Let $G_{-1}$, $G_1$, $G_0$ be vector spaces with bases $\{F_i\}$, $\{E_i\}$, $\{H_i\}$ respectively ($i=1,\ldots ,n$).
We define a structure of local Lie algebra on $\hat {G}(A):=G_{-1}\oplus G_0\oplus G_1$ by
\begin{align}
[E_i,F_j]=\delta _{ij}H_i,\quad [H_i,H_j]=0,\quad [H_i,E_j]=A_{ij}E_j,\quad [H_i,F_j]=-A_{ij}F_j.\label{eq;contra}
\end{align}
Then, we call the minimal graded Lie algebra $G(A)=\bigoplus _{i\in \Z}G_i$ with local part $\hat {G}(A)$ a contragredient Lie algebra, and the matrix $A$ its Cartan matrix.
\end{defn}
\begin{defn}[\text {reduced contragredient Lie algebras, \cite [p.1280]{ka-1}}]\label{def;redcontra}
Let $G(A)$ be a contragredient Lie algebra with Cartan matrix $A$ and $Z$ the center of $G(A)$.
We call a factor Lie algebra $G(A)/Z$ a reduced contragredient Lie algebra with Cartan matrix $A$.
\end{defn}

\section {Representations of finite-dimensional reductive Lie algebras and PC Lie algebras}
In this section, we shall give the first main theorem of this paper.
The following theorem tells us the importance of PC Lie algebras.
\begin{theo}\label {th;PC}
Let $\g $ be a finite-dimensional reductive Lie algebra, $(\rho ,V)$ a finite-dimensional completely reducible representation of $\g $, $B$ a non-degenerate symmetric invariant bilinear form on $\g $ all defined over $\C$.
Then a pentad $(\g, \rho,V,\Hom (V,\C), B)$ is standard, and moreover, the corresponding Lie algebra $L(\g,\rho,V,\Hom (V,\C),B)$ is isomorphic to some PC Lie algebra up to grading.
That is, an arbitrary finite-dimensional reductive Lie algebra and its arbitrary finite-dimensional completely reducible representation can be embedded into some PC Lie algebra.
\end{theo}
\begin{proof}
Since $\g $ is finite-dimensional, we have that the pentad $(\g,\rho,V,\Hom (V,\C),B)$ is standard (see \cite [Lemma 2.3]{Sa3}).
Denote the center part of $\g $ by $Z$ and the semisimple part of $\g $ by $\mathfrak{s}$.
Take a Cartan subalgebra $\h $ of $\mathfrak{s}$ and a fundamental system $\psi $ of the root system $R$ with respect to $(\mathfrak {s},\h )$.
If we take a non-zero root vector $X_{\gamma }$ of a root $\gamma \in R$ and define $\g $-submodules $\underline {V}\subset V$ and $\overline {\V}\subset \Hom (V,\C)$ by
$$
\underline {V}=\{y\in V\mid [X_{-\alpha },y]=0\quad \text {for any $\alpha \in \psi $}\},\quad \overline {\V}=\{\eta \in \Hom (V,\C )\mid [X_{\alpha },\eta ]=0\quad \text {for any $\alpha \in \psi $}\},
$$
then we have an isomorphism of Lie algebras up to grading:
\begin{align*}
&L(\g,\rho, V,\Hom (V,\C ),B)\\
&\quad \simeq L\left (Z\oplus \h,\rho \mid_{Z\oplus \h},\left (\sum _{\alpha \in \psi }\C X_{\alpha } \oplus \underline {V}\right ), \left (\sum _{\alpha \in \psi }\C X_{-\alpha } \oplus \overline {\V}\right ), B\mid _{(Z\oplus \h)\times (Z\oplus \h)}\right )
\end{align*}
using \cite [Theorem 3.27 and (3.16) in its proof]{Sa4}, chain rule (\cite [Theorem 3.26]{Sa3}) and the assumption that the bilinear form $B$ is symmetric.
Thus, to prove our claim, it suffices to show that the representation $\rho \mid _{Z\oplus \h}$ of $Z\oplus \h$ on $\sum _{\alpha \in \psi }\C X_{\alpha } \oplus \underline {V}$ is simultaneously diagonalizable (see \cite [Proposition 2.5]{Sa4}).
Here, note that the representation restricted to $\h $ on $\sum _{\alpha \in \psi }\C X_{\alpha } \oplus \underline {V}$ is simultaneously diagonalizable from the well-known properties of Cartan subalgebras.
Thus, if we assume that the representation $(Z\oplus \h ,\sum _{\alpha \in \psi }\C X_{\alpha } \oplus \underline {V})$ is not simultaneously diagonalizable, then there exist an element $\epsilon \in Z$, an element $a\in \C$ and an element $Y\in \underline {V}$ such that 
\begin{align}
\rho (\epsilon \otimes Y)-aY\neq 0,\quad \rho (\epsilon \otimes (\rho (\epsilon \otimes Y)-aY))-a(\rho (\epsilon \otimes Y)-aY)=0. \label{eq5}
\end{align}
Using these, define a non-zero proper vector subspace $U$ of $V$ by
$$
U=\{y\in V\mid \rho (\epsilon \otimes y)-ay=0\}.
$$
Since $\epsilon $ belongs to the center part of $\g$, $U$ is a $\g$-submodule of $V$.
Then from the assumption that the representation $\rho $ of $\g$ on $V$ is completely reducible, we have a non-zero proper $\g$-submodule $W$ of $V$ such that 
$$
V=U\oplus W.
$$
If we take elements $u\in U$ and $w\in W$ such that 
$$
Y=u+w,
$$
then we have that 
$$
0\neq \rho (\epsilon \otimes Y)-aY=(\rho (\epsilon \otimes u)-au)+(\rho (\epsilon \otimes w)-aw)=\rho (\epsilon \otimes w)-aw\in U\cap W
$$
from (\ref {eq5}).
It is a contradiction.
\end{proof}
\begin{ex}\label {ex;1}
For $m=0,1,2,\ldots $, we denote by $m\Lambda _1$ the irreducible representation of $\sl _2$ on $(m+1)$-dimensional vector space $V(m+1)$.
For example, the adjoint representation of $\sl _2$ on itself is $(\ad ,\sl _2)=(2\Lambda _1,V(3))$.
Denote the Killing form of $\sl _2$ by $K_{\sl _2}$.
Then we  have an isomorphism of Lie algebras
$$
L(\sl _2,m\Lambda _1,V(m+1),\Hom (V(m+1),\C),K_{\sl _2})\simeq L\left (1,2;\begin{pmatrix}1/8\end{pmatrix}, \begin{pmatrix}2&-m\end{pmatrix}, \begin{pmatrix}4&0\\0&4\end{pmatrix}\right )
$$
up to grading (see \cite [(3.20) in Theorem 3.28]{Sa4}).
That is, the representation $(m\Lambda _1,V(m+1))$ of $\sl _2$ can be embedded into the PC Lie algebra associated with
\begin{align}
P\left (1,2;\begin{pmatrix}1/8\end{pmatrix}, \begin{pmatrix}2&-m\end{pmatrix}, \begin{pmatrix}4&0\\0&4\end{pmatrix}\right ).\label {eq;sl}
\end{align}
\end{ex}
From Theorem \ref {th;PC}, it is natural for us to ask the structure of PC Lie algebras.

\section {Structure of PC Lie algebras}
Using the notion and notations recalled in section \ref {sec;1}, we can describe the structure of a given PC Lie algebra.
For this, we shall start with the following lemma.
\begin{lemma}\label {lem1}
We identify an $m$-tuple $(x_1,\ldots ,x_m)\in \mathfrak{gl}_1^m$ with a row vector $\begin{pmatrix}x_1&\cdots &x_m\end{pmatrix}\in \mathrm {M}(1,m;\C)$.
For an arbitrary pentad of Cartan type $P(r,n;A,D,\Gamma )$ and its corresponding Lie algebra $L(r,n;A,D,\Gamma )$, we have the following claims.
\begin{itemize}
\item [{\rm (i)}]
{We have equations
\begin{align*}
[V_{-1}, V_{1}]&=\Span (\text {the row vectors of $(\Gamma \cdot {}^t D\cdot A)$})\\
&=\left \{ \begin{pmatrix} c_1 &\cdots &c_n\end{pmatrix}\cdot \Gamma \cdot {}^{\rm t} D\cdot A\mid  c_1,\ldots ,c_n\in \C\right \}
\end{align*}
and
$$
\dim [V_{-1},V_{1}]=\rank D.
$$
}
\item [{\rm (ii)}]
{We have equations
\begin{align*}
\Ann \C_D^{\Gamma } &=\left \{c_1\epsilon _1+\cdots +c_r\epsilon _r\mid \begin{pmatrix} c_1 &\cdots &c_r\end{pmatrix}\cdot D=\begin{pmatrix} 0&\cdots &0\end{pmatrix},\quad  c_1,\ldots ,c_r\in \C \right \}\\
&=\left \{ \begin{pmatrix} c_1 &\cdots &c_r\end{pmatrix}\mid \begin{pmatrix} c_1 &\cdots &c_r\end{pmatrix}\cdot D=\begin{pmatrix} 0&\cdots &0\end{pmatrix},\quad  c_1,\ldots ,c_r\in \C \right \}
\end{align*}
and
$$
\dim \Ann \C_D^{\Gamma } =r-\rank D.
$$
}
\item [{\rm (iii)}]
{We have equations
\begin{align*}
&[V_{-1},V_1]\cap \Ann \C_D^{\Gamma } \\
&\quad =\left \{ \begin{pmatrix} c_1 &\cdots &c_n\end{pmatrix}\cdot \Gamma \cdot {}^{\rm t} D\cdot A\mid \begin{pmatrix} c_1 &\cdots &c_n\end{pmatrix}\cdot C=\begin{pmatrix} 0&\cdots &0\end{pmatrix},\quad c_1,\ldots ,c_n\in \C \right \}
\end{align*}
and
$$
\dim ([V_{-1},V_1]\cap \Ann \C_D^{\Gamma })=\rank D-\rank C,
$$
where $C=C(A,D,\Gamma )$ is the Cartan matrix of $P(r,n;A,D,\Gamma )$.
}
\end{itemize}
\end{lemma}
\begin{proof}
\begin{itemize}
\item [{\rm (i)}]
{
The vector space $[V_{-1},V_1]$ is spanned by $h_i\in \mathfrak{h}^r$ ($i=1,\ldots ,r$), which are identified with the $i$-th row vectors of the matrix $\Gamma \cdot {}^t D\cdot A$ ($i=1,\ldots ,r$) (see \cite [Proposition 2.11, Definition 2.12]{Sa4}).
Thus, we have that 
\begin{align*}
[V_{-1},V_1]=\Span (h_1,\ldots ,h_n)&=\Span (\text {the row vectors of $\Gamma \cdot {}^t D\cdot A$})\\
&=\left \{ \begin{pmatrix} c_1 &\cdots &c_n\end{pmatrix}\cdot \Gamma \cdot {}^t D\cdot A\mid  c_1,\ldots ,c_n\in \C\right \}.
\end{align*}
Moreover, since both $\Gamma \in \mathrm {M}(n,n;\C)$ and $A\in \mathrm {M}(r,r;\C)$ are invertible, we have an equation
$$
\dim [V_{-1},V_1]=\rank (\Gamma \cdot {}^t D \cdot A)=\rank D.
$$
}
\item [{\rm (ii)}]
{
This claim has been proved in \cite [Proposition 2.25]{Sa4}.
}
\item [{\rm (iii)}]
{
This claim follows from {\rm (i)} and {\rm (ii)} immediately.
}
\end{itemize}
This completes the proof.
\end{proof}
Using Lemma \ref {lem1}, we can describe the structure of an arbitrary PC Lie algebra using reduced contragredient Lie algebras.
This is the second main theorem.
\begin{theo}\label {theo;1}
Let $P(r,n;A,D,\Gamma )$ be a pentad of Cartan type and $C=C(A,D,\Gamma )=(C_{ij})_{i,j=1,\ldots , n}$ its Cartan matrix.
Let $G^{\prime }(C)$ be the reduced contragredient Lie algebra with Cartan matrix $C$. 
Then there exist vector spaces $U^{\prime }_0$, $\mathfrak{z}$, $\Delta \subset L(r,n;A,D,\Gamma )$ and a $\Z$-grading of $L(r,n;A,D,\Gamma )$
$$
L(r,n;A,D,\Gamma )=\bigoplus _{m\in \Z}U_m
$$
such that 
$$
U_0=U_0^{\prime }\oplus \mathfrak{z}\oplus \Delta ,\quad \dim \mathfrak{z}=\rank D-\rank C,\quad \dim \Delta =r-\rank D
$$
and 
\begin{align}
&U_0^{\prime }\oplus \bigoplus _{m\neq 0}U_m\simeq G^{\prime }(C), && [\mathfrak{z},L(r,n;A,D,\Gamma )]=\{0\}, &\notag \\
&[U_m,U_{-m}]\subset U_0^{\prime }\oplus \mathfrak{z},&&\text {the action of $\Delta $ on $U_m$ is diagonalizable}& \label {eqs;1}
\end{align}
for all $m\in \Z$.
\end{theo}
\begin{proof}
Let 
$$
L(r,n;A,D,\Gamma )=\bigoplus _{m\in \Z}V_m
$$
be the canonical $\Z$-grading of $L(r,n;A,D,\Gamma )$,
$$
V_{-1}\simeq \C _{-D}^{\Gamma }=\Span (f_1,\ldots ,f_n),\quad V_0\simeq \h^r=\Span (\epsilon _1,\ldots ,\epsilon _r), \quad V_1\simeq \C_D^{\Gamma }=\Span (e_1,\ldots ,e_n),
$$ and denote its bracket product by $[\cdot,\cdot ]$.
Take a complementary subspace $\Delta$ to $[V_{-1},V_1]$ in $V_0$:
$$
V_0=\mathfrak{h}^r=\mathfrak{gl}_1^r
=
[V_{-1},V_1]
\oplus \Delta .
$$
Moreover, put 
$$
\mathfrak{z}=[V_{-1},V_1]\cap \Ann \C_D^{\Gamma }
$$
and take a complementary subspace $V^{\prime }_0 $ to $\mathfrak{z}$ in $[V_{-1},V_1]$:
$$
[V_{-1},V_1]=V^{\prime }_0\oplus \mathfrak{z}=V^{\prime }_0\oplus ([V_{-1},V_1]\cap \Ann \C_D^{\Gamma }).
$$
Summarizing,
\begin{align}
V_0=[V_{-1},V_1]\oplus \Delta =V_0^{\prime }\oplus \mathfrak{z}\oplus \Delta. \label{eq;dec}
\end{align}
Then, from Lemma \ref {lem1}, we have equations:
\begin{align}
\dim V_0^{\prime }=\rank C,\quad \dim \mathfrak{z}=\rank D-\rank C, \quad \dim \Delta =r-\rank D. \label {eq3}
\end{align}
Let us denote the canonical surjection from $V_0$ to $V^{\prime }_0$ with respect to the decomposition (\ref {eq;dec}) by $p$.
Under these, to prove our claim, it is sufficient to show that a Lie algebra 
$$
L^{\prime \prime }(r,n;A,D, \Gamma )=V_0^{\prime }\oplus \bigoplus _{m\in \Z\setminus \{0\}}V_m
$$
with bracket product $[\cdot ,\cdot]^{\prime \prime}$ defined by
\begin{align*}
[x_k,y_l]^{\prime \prime }=
\begin{cases}
[x_k,y_l]&(k+l\neq 0) \\p([x_k,y_l]) &(k+l=0)
\end{cases},
\text {
where $x_k\in V_k$, $y_l\in V_l$ ($k,l\neq 0$), $x_0,y_0\in V_0^{\prime }$,
}
\end{align*}
is isomorphic to a reduced contragredient Lie algebra $G^{\prime }(C)$ with Cartan matrix $C$.
We can easily check that the bilinear map $[\cdot ,\cdot ]^{\prime \prime }$ satisfies the axioms of Lie algebras.
\par
Take elements $h_i=[e_i,f_i]\in V_0$ (see \cite [Definition 2.12]{Sa4}) and put $h^{\prime }_i=p(h_i)\in V^{\prime }_0$ for $i=1,\ldots ,n$.
Then the Lie algebra $L^{\prime \prime }(r,n;A,D, \Gamma )$ is generated by $\{f_i,h_i^{\prime },e_i\mid i=1,\ldots ,n\}$ with relations
\begin{align*}
&[h^{\prime }_i,e_j]^{\prime \prime }=C_{ij}e_j,&&[h^{\prime }_i,f_j]^{\prime \prime }=-C_{ij}f_j,&&[e_i,f_j]^{\prime \prime }=\delta _{ij}h^{\prime }_i&
\end{align*}
for all $i,j=1,\ldots ,n$ (see \cite [Proposition 2.13]{Sa4}).
On the other hand, we take $\{F_i,H_i,E_i\mid i=1,\ldots ,n\}$ a basis of $\hat {G}(C)=G_{-1}\oplus G_0\oplus G_1$, which is the local part of a contragredient Lie algebra $G(C)=G(C(A,D,\Gamma ))$, satisfying the equations (\ref {eq;contra}).
Define a linear map $\phi :\hat {G}(C)\rightarrow V_{-1}\oplus V_0^{\prime }\oplus V_1$ by
\begin{align*}
&\phi (H_i)=h^{\prime }_i,&&\phi (E_i)=e_i,&&\phi (F_i)=f_i&
\end{align*}
for $i=1,\ldots ,n$.
This linear map $\phi $ is a surjective homomorphism between the local parts of $G(C)$ and of $L^{\prime \prime }(r,n;A,D,\Gamma )$.
We can compute the kernel of $\phi $ using Lemma \ref{lem1} (iii) as follows: 
\begin{align*}
&\Ker \phi = \left \{c_1H_1+\cdots +c_nH_n\in \hat {G}(C)\mid \phi (c_1H_1+\cdots +c_nH_n)=0,\quad c_1,\ldots ,c_n\in \C  \right \} \\
&\quad = \left \{c_1H_1+\cdots +c_nH_n \in \hat {G}(C)\mid c_1h^{\prime }_1+\cdots +c_nh^{\prime }_n=0\in L^{\prime \prime }(r,n;A,D,\Gamma ), \quad c_1,\ldots ,c_n\in \C  \right \} \\
&\quad = \left \{c_1H_1+\cdots +c_nH_n \in \hat {G}(C)\mid c_1h^{\prime }_1+\cdots +c_nh^{\prime }_n\in \mathfrak{z}=[V_{-1},V_1]\cap \Ann \C_D^{\Gamma }, \quad c_1,\ldots ,c_n\in \C  \right \} \\
&\quad = \left \{c_1H_1+\cdots +c_nH_n \in \hat {G}(C)\mid c_1h_1+\cdots +c_nh_n\in \mathfrak{z}=[V_{-1},V_1]\cap \Ann \C_D^{\Gamma }, \quad c_1,\ldots ,c_n\in \C  \right \} \\
&\quad = \left \{c_1H_1+\cdots +c_nH_n \in \hat {G}(C)\mid \begin{pmatrix} c_1 &\cdots &c_n\end{pmatrix}\cdot \Gamma \cdot {}^t D\cdot A\in \mathfrak{z}, \quad c_1,\ldots ,c_n\in \C  \right \} \\
&\quad = \left \{c_1H_1+\cdots +c_nH_n \in \hat {G}(C)\mid \begin{pmatrix} c_1 &\cdots &c_n\end{pmatrix}\cdot C=\begin{pmatrix} 0&\cdots &0\end{pmatrix}, \quad c_1,\ldots ,c_n\in \C   \right \} \\
&\quad =(\text{the center of $G(C)$}).
\end{align*}
Thus, we have an isomorphism of local Lie algebras:
\begin{align*}
&\text {(the local part of $L^{\prime \prime }(r,n;A,D,\Gamma )$)}\simeq V_{-1}\oplus V_0^{\prime }\oplus V_1 \\
&\simeq \hat {G}(C)/\text {(the center of $G(C)$)} \simeq \text {(the local part of $G^{\prime }(C)$)}.
\end{align*}
Here, both graded Lie algebras $L^{\prime \prime }(r,n;A,D,\Gamma )=V_0^{\prime }\oplus \bigoplus _{m\in \Z\setminus \{0\}}V_m$ and $G^{\prime }(C)=\bigoplus _{m\in \Z}G^{\prime }_m$ are transitive except their local parts, i.e. they satisfy the condition in Definition \ref {defn;transitive} for $i\neq 0,\pm 1$.
Indeed, the transitivity for $|m|\geq 2$ of $L^{\prime \prime }(r,n;A,D,\Gamma )$ comes from the construction of the Lie algebra associated with a standard pentad (see \cite [Definitions 2.9, 2.12]{Sa3}), and, one of $G^{\prime }(C)$ comes from the construction of minimal Lie algebras (see \cite[pp.1276--1278, Proposition 4]{ka-1}).
We can extend the isomorphism between the local parts of $L^{\prime \prime }(r,n;A,D,\Gamma )$ and of $G^{\prime }(C)$ to the isomorphism between the whole graded Lie algebras:
\begin{align*}
L^{\prime \prime }(r,n;A,D,\Gamma )\simeq G^{\prime }(C)
\end{align*}
by a similar way to the proof of \cite [Theorem 1.5]{Sa4}.
Thus, we have our claim.
\end{proof}

\begin{col}
We retain to use the notations in Theorem \ref {theo;1}.
When a pentad of Cartan type $P(r,n;A,D,\Gamma )$ satisfies 
$$
r=\rank D=\rank C(A,D,\Gamma ),
$$
the corresponding Lie algebra is isomorphic to a reduced contragredient Lie algebra:
$$
L(r,n;A,D,\Gamma )\simeq G^{\prime }(C(A,D,\Gamma )).
$$
\end{col}

\begin{remark}
Under the notations in Theorem \ref {theo;1}, the vector space $\mathfrak{z}$ is contained in the center of $L(r,n;A,D,\Gamma )$.
However, in general, $\mathfrak{z}$ and the center of $L(r,n;A,D,\Gamma )$ do not coincide.
\end{remark}

We retain to use the notations in Theorem \ref {theo;1}.
When a pentad $P(r,n;A,D,\Gamma )$ has the invertible Cartan matrix, it is already shown that we have an isomorphism of Lie algebras:
$$
L(r,n;A,D,\Gamma )\simeq \gl _1^{r-n}\oplus G(C(A,D,\Gamma )) \quad \text {(see \cite [Theorem 3.9]{Sa4})}.
$$
This result is a special case of Theorem \ref {theo;1}.
In fact, from the definition of Cartan matrices of a pentad of Cartan type, we can easily show that the data satisfy conditions that
$$
r\geq n\quad \text {and} \quad \rank D=\rank C=n
$$
when $C=C(A,D,\Gamma )$ is invertible.
Under this situation, we have that 
$$
G^{\prime }(C)\simeq G(C), \quad \dim \mathfrak {z}=0, \quad \dim \Ann \C _D^{\Gamma }=r-n
$$
from Lemma \ref {lem1} and the equations (\ref {eq3}).
Since we have $\dim [V_{-1},V_1]+\dim \Ann \C _D^{\Gamma }=r=\dim V_0$ and $[V_{-1},V_1]\cap \Ann \C _D^{\Gamma }=\{0\}$, we can take $\Delta =\Ann \C_D^{\Gamma }$.
Thus, we have an isomorphism of Lie algebras:
$$
L(r,n;A,D,\Gamma )\simeq G^{\prime }(C)\oplus \mathfrak{z}\oplus \Delta \simeq G(C)\oplus \{0\}\oplus \Ann \C _D^{\Gamma } \simeq \gl _1^{r-n}\oplus G(C).
$$
\begin{lemma}\label {lem2}
Let $P(r,n;A,D,\Gamma )$ be a pentad of Cartan type.
Assume that the $i$-th column vector of $D$ is a zero-vector.
Then corresponding elements $e_i\in \C_D^{\Gamma }$ and $f_{i}\in \C_{-D}^{\Gamma }$ belong to the center of $L(r,n;A,D,\Gamma )$.
\end{lemma}
\begin{proof}
Denote $L(r,n;A,D,\Gamma )=\bigoplus _{n\in \Z}V_n$ and identify $V_{-1},V_0,V_1$ with $\C_{-D}^{\Gamma }, \h^r, \C_D^{\Gamma }$ respectively.
Let us show that $e_i$ belongs to the center of $L(r,n;A,D,\Gamma )$.
From the assumption of our claim, it is clear that $[V_0,e_i]=\{0\}$.
Take arbitrary elements $a\in V_0$ and $f\in V_{-1}$.
Then we have an equation
$$
B_A(a,[e_i,f])=B_A(a,\Phi _{\Box_D^r}(e_i\otimes f))=\langle \Box _D^r(a\otimes e_i),f\rangle ^{\Gamma }_D=\langle 0,f\rangle ^{\Gamma }_D=0.
$$
Since $B_A$ is non-degenerate on $V_0\simeq \h^r$, we have that $[V_{-1},e_i]=\{0\}$.
Moreover, it holds that $[V_1,e_i]=\{0\}$.
In fact, for any $f\in V_{-1}$, we have 
$$
[[V_1,e_i],f]\subset [[V_1,f],e_i]+[V_1,[e_i,f]]\subset [V_0,e_i]+[V_1,0]=\{0\}.
$$
Since $[V_1,V_1]\subset V_2\subset \Hom (V_{-1},V_1)$ (see \cite [Definitions 2.9, 2.12]{Sa3}), we have that $[V_1,e_i]=\{0\}$.
By a similar argument, we have the same results on $f_i$:
$$
[V_{-1}\oplus V_0\oplus V_1, e_i]=[V_{-1}\oplus V_0\oplus V_1, f_i]=\{0\}.
$$
Since $L(r,n;A,D,\Gamma )$ is generated by $V_{-1}\oplus V_0\oplus V_1$, we have our result.
\end{proof}
From Lemma \ref {lem2}, we have the following claim immediately.
\begin{lemma}\label {lem3}
Let $P(r,n;A,D,\Gamma )$ be a pentad of Cartan type and assume that $D$ and $\Gamma $ are of the forms 
$$
D=
\left (
\begin{array}{c|c}
D^{\prime }&O
\end{array}
\right ),
\ 
\Gamma =
\left (
\begin{array}{c|c}
\Gamma ^{\prime }&O\\ \hline
O&\Gamma ^{\prime \prime }
\end{array}
\right )
$$
for some $D^{\prime }\in \mathrm {M}(r,n^{\prime };\C),
\  
\Gamma ^{\prime }\in \mathrm {M}(n^{\prime },n^{\prime };\C),
\  
\Gamma ^{\prime \prime }\in \mathrm {M}(n-n^{\prime },n-n^{\prime };\C)$.
Then we have an isomorphism of Lie algebras:
$$
L(r,n;A,D,\Gamma )\simeq \gl _1^{2(n-n^{\prime })}\oplus L(r,n^{\prime };A,D^{\prime },\Gamma ^{\prime } )
$$
up to grading.
\end{lemma}
From Lemma \ref  {lem3}, to simplify the calculation, we can assume that $D$ does not have zero-column vectors without loss of generality.
\begin{ex}
We retain to use the notations in Example \ref {ex;1}.
Let us find the structure of $L(\sl _2,m\Lambda _1,V(m+1),\Hom (V(m+1),\C ),K_{\sl _2})$ from the pentad (\ref  {eq;sl}).
For this, we need to calculate the Cartan matrix of (\ref {eq;sl}):
$$
C\left (
\begin{pmatrix}1/8\end{pmatrix}, \begin{pmatrix}2&-m\end{pmatrix}, \begin{pmatrix}4&0\\0&4\end{pmatrix}
\right )
= \begin{pmatrix}4&0\\0&4\end{pmatrix}\cdot \begin{pmatrix}2\\-m\end{pmatrix}\cdot \begin{pmatrix}1/8\end{pmatrix}\cdot \begin{pmatrix}2&-m\end{pmatrix}=\begin{pmatrix}2&-m\\-m&m^2/2\end{pmatrix}.
$$
Under the notations of Theorem \ref {theo;1}, we have equations
\begin{align*}
&\dim \mathfrak{z}=\rank \begin{pmatrix}2&-m\\-m&m^2/2\end{pmatrix}-\rank \begin{pmatrix}2&-m\end{pmatrix}=1-1=0,\\
&\dim \Delta =1-\rank \begin{pmatrix}2&-m\end{pmatrix}=1-1=0
\end{align*}
for any integer $m\geq 0$.
Thus, we have an isomorphism of Lie algebras:
$$
L(\sl _2,m\Lambda _1,V(m+1),\Hom (V(m+1),\C ),K_{\sl _2})\simeq G^{\prime }\left (\begin{pmatrix}2&-m\\-m&m^2/2\end{pmatrix}\right ).
$$
That is, the representation $(m\Lambda _1,V(m+1))$ of $\sl _2$ can be embedded into the reduced contragredient Lie algebra
$$
G^{\prime }\left (\begin{pmatrix}2&-m\\-m&m^2/2\end{pmatrix}\right ).
$$
\end{ex}

\medskip
\begin{flushleft}
Nagatoshi Sasano\\
Institute of Mathematics-for-Industry\\
Kyushu University\\
744, Motooka, Nishi-ku, Fukuoka 819-0395\\
Japan\\
E-mail: n-sasano@math.kyushu-u.ac.jp
\end{flushleft}

\end{document}